\documentclass[12pt]{amsart}
\usepackage{amscd,amssymb,graphics}

\usepackage{amsfonts}
\usepackage{amsmath}
\usepackage{amsxtra}
\usepackage{latexsym}
\usepackage[mathcal]{eucal}

\usepackage{graphics,colortbl}

\input xy
\xyoption{all}
\usepackage{epsfig}

\usepackage[pdftex,bookmarks,colorlinks,breaklinks]{hyperref}
\makeatletter
\@namedef{subjclassname@2020}{%
  \textup{2020} Mathematics Subject Classification}
\makeatother
\newtheorem{theorem}{Theorem}[section]
\newtheorem{fact}[theorem]{Fact}
\newtheorem{corollary}[theorem]{Corollary}
\newtheorem{lemma}[theorem]{Lemma}
\newtheorem{proposition}[theorem]{Proposition}

\newtheorem{question}[theorem]{Question}

\numberwithin{equation}{section}

\newtheorem{claim}{Claim}

\theoremstyle{definition}

\newtheorem{definition}[theorem]{Definition}

\newcommand{\ben}{\begin{enumerate}}
	\newcommand{\een}{\end{enumerate}}
\newcommand{\bit}{\begin{itemize}}
	\newcommand{\eit}{\end{itemize}}

\def\di{{\mathrm{di}}}

\def\QED{\nobreak\quad\ifmmode\roman{Q.E.D.}\else{\rm Q.E.D.}\fi}

\def \di {{\mathrm{d}}}

\def\GL{\operatorname{GL}}

\DeclareMathOperator{\res}{\restriction}

	\title[Successors of locally compact topologies]{Successors of topologies of connected locally compact groups}

\author[D. Peng, Z. Xiao]{Dekui Peng, Zhiqiang Xiao*}
   \address[D. Peng]
	{\hfill\break Institute of Mathematics,
		\hfill\break Nanjing Normal University, 210023,
		\hfill\break China}
\email{pengdk10@lzu.edu.cn}

	\address[Z. Xiao]
	{\hfill\break Department of Mathematics,
		\hfill\break Taizhou University, 225300,
		\hfill\break China}
	\email{zqxiao@bicmr.pku.edu.cn}
\thanks{*Corresponding Author}

	
\begin{document}
	\maketitle	
\thanks{\begin{center}Dedicated to 60th anniversary of Professor Wei He\end{center}}
	\begin{abstract} Let $G$ be a group and $\sigma, \tau$ be  topological group topologies on $G$. We say that $\sigma$ is a successor of $\tau$ if $\sigma$ is strictly finer than $\tau$ and there is not a group topology properly between them.
	In this note, we explore the existence of successor topologies in topological groups, particularly focusing on non-abelian connected locally compact groups. Our main contributions are twofold: for a connected locally compact group $(G, \tau)$, we show that (1) if $(G, \tau)$ is compact, then $\tau$ has a precompact successor if and only if there exists a discontinuous homomorphism from $G$ into a simple connected compact group with dense image, and (2) if
$G$ is solvable, then
$\tau$ has no successors.
Our work relies on the previous characterization of locally compact group topologies on abelian groups processing successors.
\end{abstract}

\section{Introduction}

\subsection{The lattice of group topologies}

Let $G$ be a group. The set $\mathcal{L}(G)$ of all (not necessarily Hausdorff) group topologies on $G$ forms a complete lattice, with the natural inclusion order.
This lattice has been studied by many authors. For example, if $G$ is abelian, then $\mathcal{L}(G)$ has been shown in \cite{Kil} to contain $2^{2^{|G|}}$ many distinct Hausdorff group topologies, the maximum possible number.
This result was generalized by Dikranjan and Protasov \cite{DP}, who showed that $\mathcal{L}(G)$ contains $2^{2^{|G|}}$ many {\bf maximal} Hausdorff group topologies.
Recall that a group topology is called {\em maximal} if it is succeeded by the discrete topology, see Definition \ref{gap} below.
Let $\mathcal{L}_2(G)$ be the subset of $\mathcal{L}(G)$ consisting of Hausdorff group topologies. Minimal elements in this set are called {\em minimal group topologies}. A group with such a topology is called a {\em minimal topological group.} Minimal topological groups are one of the most widely studied topics in the field of topological groups, we refer the readers to the nice book \cite{DPS89} for more about minimal topological groups.

Let us introduce another object in $\mathcal{L}(G)$ recently studied by the authors and collaborators, which has been shown in \cite{HP0, HP} to have deep connections with minimal topological groups.
\begin{definition}\label{gap}
Let $G$ be a group and $\tau, \sigma\in \mathcal{L}(G)$, if $\tau\subsetneq \sigma$ and there is no $\lambda\in \mathcal{L}(G)$ satisfying that $\tau\subsetneq \lambda\subsetneq \sigma$, then we call the pair $\{\tau, \sigma\}$ a {\em gap} \cite{HPTX1, HPTX2, PHTX} or a {\em cover} \cite{Ar2} (in $\mathcal{L}(G)$).
 We also say that $\sigma$ is a {\em successor} of $\tau$ or $\sigma$ {\em succeeds} $\tau$, in symbol $\tau\prec \sigma$.
We write $\tau\preceq \sigma$ to mean ``$\tau=\sigma$ or $\tau\prec \sigma$.''
\end{definition}
We say that $\tau$ is a predecessor of $\sigma$ if $\tau\prec \sigma$.

The study of gaps is closed related to the (semi-)modularity of the lattice $\mathcal{L}(G)$, as semi-modular lattices satisfy the so-called ``(upper) cover condition'', that is $a\prec b$ implies $a\vee c\preceq b\vee c$ for any $c$, see \cite{Gr}.
A very useful result is that the lattice $\mathcal{L}(G)$ is modular whenever $G$ is abelian \cite{Sm}. By the help of the nice property,
the authors and collaborators in \cite{PHTX} provided the following characterization of compact abelian groups\footnote{locally compact spaces are always assumed to be Hausdorff} whose topologies have successors:
\begin{fact}\label{AbelCom}
Let $(G, \tau)$ be a compact abelian group. Then the following conditions are equivalent:
\begin{itemize}
   \item[(1)] $\tau$ has a successor;
   \item[(2)] $G/pG$ is infinite for some prime number $p$;
   \item[(3)] $(G, \tau)$ has a dense subgroup of prime index.
\end{itemize}
\end{fact}
Since connected compact abelian groups are divisible, $G/pG$ is never infinite when $G$ is connected.
Consequently, connected compact group topologies on abelian groups do not have successors.
From this, we derive the following corollary:
\begin{fact}\label{AbelLC}\cite{PHTX}
Let $(G, \tau)$ be a connected locally compact abelian group. Then $\tau$ has no successors.
\end{fact}
In this note, we continue our investigation into the existence of successors, with a particular focus on non-abelian connected locally compact groups.
Since generally the lattice $\mathcal{L}(G)$ is not modular, the study is much more difficult.

The main results of this note are as follows: Let $(G, \tau)$ be a connected locally compact group.
\begin{itemize}
  \item (Theorem \ref{main1}) If $(G, \tau)$ is compact, then $\tau$ admits a precompact successor if and only if $(G, \tau)$ admits a discontinuous homomorphism into a simple connected compact group with a dense image.
  \item (Theorem \ref{solLC}) If $G$ is solvable, then $\tau$ admits no successors.
\end{itemize}

We would like to note that in \cite{XHP}, the authors proved that the topology of locally compact minimal group $\mathbb{R}\rtimes \mathbb{R}^+$ has no successors, where $\mathbb{R}^+$ is the multiplicative group of positive real numbers and the action is non-trivial. Since this group is connected and solvable, it is an immediate corollary of our new result.

\subsection{Subgroup and quotient group topologies}
Let $(G, \tau)$ be a topological group and $N$ a closed subgroup. We denote by $\tau\res_N$ the subspace topology of $N$ and by $\tau/N$ the quotient topology of the coset space $G/N$.
The following result, known as Merson's Lemma, plays a crucial role in the study of successors as well as minimal topological groups:
\begin{fact}\label{Mer}\cite[Lemma 7.2.3]{DPS89}
Let $G$ be a group and $N$ a subgroup. If two group topologies $\tau$ and $\sigma$ of $G$ satisfy
\begin{itemize}
  \item[(1)] $\tau\subseteq \sigma$;
  \item[(2)] $\tau\res_N=\sigma\res_N$; and
  \item[(3)] $\tau/N=\sigma/N$,
\end{itemize}
then they coincide.
\end{fact}

Furthermore, we have the following elementary fact regarding closed normal subgroups:

\begin{fact}\label{sub}\cite[Lemma 3.25]{HPTX2}
Let $G$ be a group and $\tau, \sigma$ be group topologies on $G$ such that $\tau\prec \sigma$.
If $N$ is a normal subgroup of $G$, then
\begin{itemize}
\item[(1)] $\tau/N\preceq \sigma/N$;
\item[(2)] $\tau\res_N\preceq \sigma\res_N$ when $N$ is additionally assumed to be central, i.e., contained in the centre of $G$.
\end{itemize}
\end{fact}
This leads us to the following corollary:

\begin{corollary}\label{cent}
Let $G$ be a group and let $\tau$ and $\sigma$ be group topologies on $G$ such that $\tau \subseteq \sigma$. If $N$ is a normal subgroup of $G$, then $\tau \prec \sigma$ under either of the following conditions:
\begin{itemize}
\item[(1)] $\tau \res_N = \sigma \res_N$ and $\tau/N \prec \sigma/N$;
\item[(2)] $\tau \res_N \prec \sigma \res_N$ and $\tau/N = \sigma/N$.
\end{itemize}
Moreover, if $N$ is central, then this condition is also necessary for $\tau \prec \sigma$.
\end{corollary}
\begin{proof}
 Suppose (1) holds and take a group topology $\lambda$ on $G$ with $\tau\subset \lambda\subsetneqq \sigma$.
Then $$\tau\res_N=\lambda\res_N=\sigma\res_N.$$
Moreover, we have that $\tau/N\subseteq \lambda/N\subseteq \sigma/N$.
Since $\lambda\neq \sigma$, Merson's Lemma then implies that $\lambda/N\neq \sigma/N$.
And since $\tau/N\prec \sigma/N$, we obtain that
$$\tau/N=\lambda/N.$$
Using Merson's Lemma again, one has $\lambda=\tau$.
Thus $\tau\prec \sigma$.
The case (2) follows from a similar argument, we omit the proof.

To see the second assertion, let $\tau\prec \sigma$. Then by Fact \ref{sub}, we have that $\tau\res_N\preceq \sigma\res_N$ and $\tau/N\prec \sigma/N$.
Then, according to Merson's Lemma, the two ``='' cannot hold at the same time since $\tau\neq \sigma$.
\end{proof}

If $(G, \tau)$ is a topological group, we will denote by $\mathcal{N}_\tau(G)$ the filter of neighbourhoods of the identity.
Now, consider the following corollary concerning neighborhood filters, which provides insight into the behavior of successors in quotient groups:
\begin{corollary}\label{quot}
Let $(G, \tau)$ be a topological group and $N$ a normal subgroup. Suppose $\lambda$ is a group topology on $G/N$ succeeding $\tau/N$.
Then the group topology $\sigma$ on $G$, which has a neighbourhood base of
$$\{U\cap \pi^{-1}(V): U\in \mathcal{N}_{\tau}(G), V\in \mathcal{N}_{\lambda}(G/N)\},$$
is a successor of $\tau$, where $\pi$ denotes the quotient homomorphism $G\to G/N$.
In particular, if $\tau$ has no successors, then neither does $\tau/N$.
\end{corollary}
\begin{proof}
From the construction of the topology $\sigma$, one can easily check that
$$\sigma\res_N=\tau\res_N~\mbox{~and~}~\sigma/N=\lambda\succ \tau/N.$$
Then by Corollary \ref{cent}, $\tau\prec \sigma$.
\end{proof}
\subsection{Commutator subgroups and solvable groups}
We compile several established facts about solvable topological groups in this subsection.
For an abstract group $G$, an element is termed a {\em commutator} if it can be expressed as $ghg^{-1}h^{-1}$ for some $g,h\in G$.
The subgroup of $G$ generated by all commutators, denoted $G'$ or $[G, G]$, is called the {\em commutator subgroup} of $G$.
In the context of a topological group $G$, $\overline{G'}$, the closure of $G'$, is referred to as the {\em topological commutator subgroup} of $G$.
Generally, $G'$ and $\overline{G'}$ are distinct subgroups of $G$. However, the following result by Goto is significant:
\begin{fact}\label{goto}\cite[Theorem 9.2]{HM}
Let $G$ be a connected compact group. Then $G'$ is closed in $G$ and every element of $G'$ is a commutator.
\end{fact}

For a group $G$, we define $G^{(0)}=G$ and recursively $G^{(n+1)}=[G^{(n)}, G^{(n)}]$, for $n\geq 0$. Since each $G^{(n+1)}$ is a characteristic subgroup of $G$, meaning it is invariant under any automorphism of $G$, every $G^{(n)}$ is a normal subgroup of $G$.
The group $G$ is termed {\em solvable} if $G^{(n)}=\{1\}$ for some integer $n$.

\begin{proposition}\label{abelnorm}
Let $G$ be a connected solvable topological group. Then $G$ has a non-trivial closed connected normal abelian subgroup.
\end{proposition}
\begin{proof}
Let $G$ be a connected solvable topological group. Define its derived series as:
$$G=G^{(0)}\geq G^{(1)}\geq G^{(2)}\geq...\geq G^{(n)}=\{1\},$$
where $G^{(k+1)}=[G^{k}, G^{k}]$~for $k\geq 0$.
 Notice that
$G^{(n-1)}$ is an abelian subgroup of $G$, we need to demonstrate that
it is also  connected.

The connectedness of $G^{(0)}$ follows from our assumption that $G$ is connected.
 Suppose $G^{(k-1)}$
  is connected for some $k\geq 1$. We will show by induction that $G^{(k)}$
  is also connected.
Consider the set $X:=\{xyx^{-1}y^{-1}| x,y\in G^{(k-1)}\}$, which is connected because $G^{(k-1)}$ is connected. Then
$$G^{(k)}=[G^{(k-1)},G^{(k-1)}]=\bigcup_{i=1}^\infty X^i,$$
where $X^i$ is defined as the set $\{x_1x_2...x_i: x_1, x_2,...,x_i\in X\}$.
Since $X^i$ contains the identity, the chain $\{X^1, X^2,..., X^i,...\}$ is increasing.
As a continuous image of $\underbrace{X\times X\times...\times X}_{i~\text{copies}}$, $X^i$ is connected.
So, $G^{(k)}$ is connected.
In particular, $G^{(n-1)}$ is connected.
So its closure is a required group.
Thus, we have shown that a connected solvable topological group  always possesses a non-trivial closed connected normal abelian subgroup.
\end{proof}

\section{Main Results}
From here on, unless otherwise stated, topological groups are always assumed to be Hausdorff.

\begin{lemma}\label{absub}
Let $(G, \tau)$ be a connected locally compact group and $A$ a compact normal abelian subgroup. Then any successor $\sigma$ of $\tau$ coincides with $\tau$ on $A$.
\end{lemma}
\begin{proof}
By a result of Iwasawa \cite[Theorem 4]{Iwa}, a compact normal abelian subgroup of a connected topological group must be central. Hence, by Fact \ref{sub}, if $\sigma\res_A \neq \tau\res_A$, then $\sigma/A = \tau/A$ and $\sigma\res_A$ is a successor of $\tau\res_A$.
We claim that that $A$ must be contained in a connected abelian subgroup $B$ of $G$.
 According to \cite[Theorem 13]{Iwa}, $A$ is contained in a connected compact subgroup $N$ of $(G, \tau)$.
 By \cite[Theorem 9.32]{HM}, every abelian subgroup of a connected compact group is contained in a connected one.
 So the assertion holds.

Let $\tau_B = \tau\res_B$ and $\sigma_B = \sigma\res_B$. From this setup, we deduce that $$\sigma_B\res_A = \sigma\res_A \succ \tau\res_A = \tau_B\res_A$$ and $$\sigma_B/A = (\sigma/A)\res_{B/A} = (\tau/A)\res_{B/A} = \tau_B/A.$$
 Consequently, by Corollary \ref{cent}, $\tau_B \prec \sigma_B$. This, however, contradicts Fact \ref{AbelLC}, given that $(B, \tau_B)$ is a connected compact abelian group.
\end{proof}
Now, using Corollary \ref{cent}, the following is evident.
\begin{corollary}\label{connsub}
For a connected locally compact group $(G, \tau)$ and a compact normal abelian subgroup $A$, any successor $\sigma$ of $\tau$ ensures that $\sigma/A$ succeeds $\tau/A$.
\end{corollary}
Recall that a connected compact group is called {\em semisimple} if its centre is totally disconnected \cite{HM}.
We will see that when we study the existence of successors of a connected compact group topology, we can only consider the ``semisimple part''.
\begin{corollary}\label{comm}
For a connected compact group $(G, \tau)$, any successor $\sigma$ of $\tau$ ensures that $\sigma\res_{G'}$ is a successor of $\tau\res_{G'}$, where $G'$ is the commutator subgroup of $G$.
\end{corollary}
\begin{proof}
Recall that $G'$ is a connected closed subgroup of $G$ (see Fact \ref{goto}). Let $Z_0$ be the identity component of the centre of $G$. According to \cite[Theorem 9.24]{HM}, $G = Z_0 G'$ and $\Delta := Z_0 \cap G'$ is a totally disconnected central subgroup of $G$. Let $\tau' = \tau\res_{G'}$ and $\sigma' = \sigma\res_{G'}$. By Lemma \ref{absub}, we have that

\begin{equation}\label{e1}\sigma\res_{\Delta} = \tau\res_{\Delta}\end{equation}
 and
\begin{equation}\label{e2}\sigma\res_{Z_0} = \tau\res_{Z_0}.\end{equation}

We then derive from (\ref{e1}) that
\begin{equation}\label{e3}\sigma'\res_\Delta = \sigma\res_\Delta = \tau\res_\Delta = \tau'\res_\Delta.\end{equation}

 Furthermore, by (\ref{e2}) and Corollary \ref{cent}, we obtain that
\begin{equation}\label{e4}\sigma/Z_0 \succ \tau/Z_0.\end{equation}

 Now let $\lambda$ be either $\sigma$ or $\tau$ and consider the canonical mapping
 $$\pi: (G, \lambda)\to (G/Z_0, \lambda/Z_0).$$
 By (\ref{e2}), $\lambda\res_{Z_0}$ is compact. So $\pi$ is a closed mapping. Since $G'$ is $\tau$-closed (hence, $\lambda$-closed), the restriction of $\pi$ to $G'$ remains closed.
 Then it follows from the equality $G=Z_0G'$ that $$(G'/\Delta, \lambda/\Delta)\cong\pi(G')=(G/Z_0, \lambda/Z_0).$$
 So by (\ref{e4}), we have that
 \begin{equation}\label{e5}\sigma'/\Delta \succ \tau'/\Delta.\end{equation}
 In summary, (\ref{e3}) together with (\ref{e5}) imply that $\sigma'\succ \tau'$, by Corollary \ref{cent}.
\end{proof}
Another important result is the following lemma:
\begin{lemma}\label{totcent}
Let $G$ be a connected topological group with $N$ being a totally disconnected closed normal subgroup. Then $N$ is central.
\end{lemma}
\begin{proof}
The action $\varphi: G\times N\to N, (g,x)\mapsto gxg^{-1}$ is continuous.
Therefore for any $x\in N$, its orbit, as a continuous image of $G$, must be connected. Since $N$ is totally disconnected, this orbit can only be the singleton $\{x\}$.
This implies that the action $\varphi$ is trivial. Hence, $N$ is in the centre of $G$.
\end{proof}
\begin{proposition}
Let $(G, \tau)$ be a connected compact group and $G'$ the commutator subgroup. Then $\tau$ has a successor if and only if $\tau\res_{G'}$ has a successor.
\end{proposition}
\begin{proof}
The ``only if'' part directly follows from Corollary \ref{comm}.

Now suppose that $\tau' := \tau\res_{G'}$ has a successor $\sigma'$.
 Again, let $Z_0$ be the identity component of the centre of $(G, \tau)$.
 Then $G = Z_0 G'$ and $Z_0 \cap G'$ is totally disconnected and hence central (see Lemma \ref{totcent}). Let $\varphi: Z_0 \times G' \to G$, naturally mapping $(z, x)$ to $zx$.
 The kernel $C$ of $\varphi$ is $\{(z, z^{-1}) : z \in Z_0 \cap G' \}$. Being topologically isomorphic to $Z_0 \cap G'$, $C$ is totally disconnected and central.
 Moreover,  $\varphi$ operates openly when $Z_0, G'$, and $G$ carry the topology induced from $\tau$.
 If we identify $G$ as $(Z_0 \times G')/C$, then $\tau = (\tau_0 \times \tau')/C$, where $\tau_0 = \tau\res_{Z_0}$.

Considering $\tau_0 \times \sigma'$ on $Z_0 \times G'$, it demonstrably coincides with $\tau_0 \times \tau'$ when restricted to $Z_0$ (while naturally aligning $Z_0$ with a subgroup of $Z_0 \times G'$).
Furthermore, we have that
$$(\tau_0 \times \sigma')/Z_0 = \sigma' \succ \tau' = (\tau_0 \times \tau')/Z_0.$$
 Thereby, according to Corollary \ref{cent}, $\tau_0 \times \tau' \prec \tau_0 \times \sigma'$. By applying Corollary \ref{connsub}, we obtain that $\tau = (\tau_0 \times \tau')/C \prec (\tau_0 \times \sigma')/C$, indicating that $\tau$ possesses a successor.
\end{proof}

\subsection{Precompact successors of connected compact group topologies}
For a precompact group, we mean a topological group $G$ such that for every identity neighbourhood $U$ there is a finite subset $F$ of $G$ such that $G=FU$.
It is well-known that a topological group is precompact if and only if its Weil completion (or, equivalently, Rajkov completion) $\varrho G$ is a compact group.
Moreover, if $G\to H$ is a continuous homomorphism of precompact groups, then there exists a (necessarily unique and surjective) continuous homomorphism $\varrho G\to \varrho H$ extending it.

In the context of precompact group topologies, it is useful to characterize when one such topology is a successor of another. This leads us to the following lemma, which provides a necessary and sufficient condition for this relationship:
\begin{lemma}\label{Crit}\cite[Theorem 4.2]{HP}
Let $G$ be a group and $\tau, \sigma$ be precompact group topologies on $G$ such that $\tau\subsetneq \sigma$.
Then $\sigma$ is a successor of $\tau$ if and only if the following holds:

The kernel of the continuous extension of identity mapping $(G, \sigma)\to (G, \tau)$ is a minimal closed normal subgroup of $\varrho(G, \sigma)$.
\end{lemma}
Here, a {\em minimal} closed normal subgroup of a topological group is a non-trivial closed normal subgroup which does not contain any other non-trivial closed normal subgroup.

A group $G$ is called {\em divisible} if for any $g\in G$ and every positive $n\in \mathbb{Z}$, the equation $x^n=g$ has a solution in $G$.
A consequence of Pontryagin Duality tells that a compact abelian group is divisible if and only if it is connected.
Furthermore,  since every element in a connected compact group is contained in a connected compact abelian subgroup, it follows that connected compact groups are divisible. Conversely, a disconnected compact group always admits a continuous homomorphism onto a finite group, and hence cannot be divisible. In summary, a compact group is divisible if and only if it is connected, regardless of whether it is abelian or not.

The following lemma is also well-known and frequently used in the study of topological groups:

\begin{lemma}\label{div}
Let $G$ be a topological group with a dense divisible subgroup $H$. If $G$ is compact, then it is divisible.
\end{lemma}
\begin{proof}
For a fixed positive integer $n$ the mapping
$$\varphi: G\to G, g\mapsto g^n$$
is continuous.
So $\varphi(G)$ is compact.
Since $\varphi(G)$ contains $\varphi(H)=H$, we obtain that $\varphi(G)=G$.
Thus, $G$ is divisible.
\end{proof}

Having established these foundational lemmas, we now proceed to the main result concerning the existence of precompact successors in the context of connected compact groups.

Recall that a {\em simple connected compact Lie group} is a centre-free connected compact Lie group whose Lie algebra is simple. It is well-known (and easy to check) that a connected compact group is (topologically) simple if and only if it is a simple connected compact Lie group.

\begin{theorem}\label{main1}
Let $(G, \tau)$ be a connected compact group. Then $\tau$ has a precompact successor if and only if there is a discontinuous homomorphism $(G, \tau)\to H$ with dense image, where $H$ is a simple connected compact Lie group.
\end{theorem}
\begin{proof}
First, we assume that $(G, \tau)$ admits such a homomorphism $f: (G, \tau)\to H$.
Let $\sigma$ be the coarsest topology on $G$ which is finer than $\tau$ and makes $f$ continuous.
Then the mapping $$\varphi: (G, \sigma)\to (G, \tau)\times H, x\mapsto (x, f(x))$$ is a topological group embedding.
\begin{claim}\label{claim1}  $G^*:=\varphi(G)$ is dense in $G\times H$.\end{claim}
\begin{proof}[Proof of Claim \ref{claim1}.]
It is evident that $G^*$ cannot be closed in $G\times H$ since it is not compact.
Let $C$ be the closure of $G^*$ and $\pi: G\times H\to G$ the projection.
Since $C$ properly contains $G^*$, there exists $x\in C\setminus G^*$.
Note that $\pi(G^*)=G$, so we can take $y\in G^*$ such that $\pi(y)=\pi(x)$.
Thus, $(1,1)\neq x^{-1}y\in C\cap (\{1\}\times H).$
In other words, $C$ non-trivially meets $\{1\}\times H$.

Let $\{1\}\times N$ be the intersection of $C$ and $\{1\}\times H$, where $N$ is a non-trivial closed subgroup of $H$.
Then $\{1\}\times N$ is a closed normal subgroup of $C$.
In particular, $\{1\}\times N$ is normalized by $G^*$.
Since $G^*=\{(g, f(g)): g\in G\}$, $N$ is normalized by $f(g)$ for any $g\in G$.
Note that the normalizer of $N$ in $H$ should be a closed subgroup of $H$.
Then by the denseness of $f(G)$ in $H$, we obtain that $N$ is normal in $H$.
By our assumption that $H$ is simple, the non-trivial closed normal subgroup $N$ of $H$ must coincide with $H$.
Thus, $\{1\}\times H\subseteq C$.
It follows from $\pi(C)=G=\pi(G\times H)$ that $C=G\times H$.
That is, $G^*$ is dense in $G\times H$.
\end{proof}

Now, let us identify $(G, \sigma)$ with $G^*$. Then $\varrho (G,\sigma)=(G, \tau)\times H$ and the projection $\pi: (G, \tau)\times H\to G$ is the extension of the identity mapping $(G, \sigma)\to (G, \tau)$.
The kernel of $\pi$ is topologically isomorphic to $H$, which is a simple, hence minimal, closed normal subgroup of $(G, \tau)\times H$.
By Lemma \ref{Crit}, $\sigma$ is a successor of $\tau$.

Let us consider the conversion. Suppose that $\sigma$ is a precompact successor of $\tau$ and $Z$ is the centre of $G$.
By Corollary \ref{connsub}, we have $\tau/Z\prec \sigma/Z$.
Moreover, by the structure theorem of connected compact groups \cite[Theorem 9.24]{HM}, we know that $(G/Z, \tau/Z)$ is topologically isomorphic to a direct product of simple connected compact Lie groups.
So, we may assume without loss of generality that $(G, \tau)=\prod_{i\in I}S_i$, where each $S_i$ is a simple connected compact Lie group.
In particular, $G$ is centre-free.

The extension of the identity mapping $\varphi: (G, \sigma)\to (G, \tau)$ over $K:=\varrho(G, \sigma)$ is denoted by $\pi$.
According to Lemma \ref{Crit}, $N:=\ker \pi$ is a minimal closed normal subgroup of $K$.
We write the dense subgroup $(G, \sigma)$ of $K$ as $G^*$; then $K=NG^*$ and $N\cap G^*$ is trivial. So, algebraically, $K=N\rtimes G^*$.

\begin{claim}\label{claim2} $K$ is centre-free.
\end{claim}
\begin{proof}[Proof of Claim \ref{claim2}]
Since $G$ is divisible and $K$ is compact, it follows from Lemma \ref{div} that $K$ is also divisible, hence connected.

Let $Z$ be the centre of $K$. Since $G$ is centre-free, $Z$ must be contained in the kernel of $\pi: K\to (G, \tau)$, $N$.
Being a minimal closed normal subgroup of $K$, $N$ must coincide with $Z$ if $Z$ is non-trivial.
So, the semi-direct product $K=NG^*$ is direct.
Since every subgroup of $Z$ is normal in $K$, $Z$ cannot have any non-trivial subgroup but $Z$.
Therefore, $Z$ is a cyclic group of prime order.
If this were true, then  $K$ would not be divisible, because $N$, as a (discontinuous) homomorphism image of $K$, is not divisible.
This contradiction completes the proof of Claim 2.
\end{proof}

By the above Claim 2, using the structure theorem \cite[Theorem 9.24]{HM} again (or \cite[Theorem 9.19]{HM}), we obtain that $K$ is also of the form
$\prod_{j\in J} T_j$, where each $T_j$ is a simple connected compact Lie group.
Moreover, a closed subgroup $M$ of $K$ is normal if and only if it is a {\em partial product}, i.e., there exists a subset $J'$ of $J$ such that $M=\prod_{j\in J'} T_j$ (here we identity the latter group as a subgroup of $\prod_{j\in J}T_j$ in a natural way).
This tells us that there exists $k\in J$ such that $N=T_k$.
Let $S=\prod_{j\in J\setminus \{k\}} T_j$, then $K=N\times S$.
Thus, $S\cong K/N\cong (G, \tau)$.
Denote by $p$ the natural projection $K\to N$ with kernel $S$.
Then we obtain a homomorphism $q:=p\circ\varphi^{-1}$ of $(G, \tau)$ to $N$.
We shall see that $q$ is a desired mapping.

From $q(G)=p(G^*)$ and the denseness of $G^*$ in $K$, we know that $q(G)$ is dense in $N$.
It remains to see that $q$ is discontinuous.
To establish the discontinuity of $q$, suppose $q$ were continuous. Then $q$ would be an open surjective map due to the compactness of $(G, \tau)$.
Consider an element $x$ of $G$ distinct from the identity.
Then there exists an identity neighbourhood $U$ in $(G, \tau)$ such that $U\cap xU=\emptyset$.
Let $W=\pi^{-1}(U)\cap p^{-1}(q(xU))$.
Choose $y\in q(U)\subseteq N$ and $z\in S$ such that $\pi(z)\in x^{-1}U$. Set $g=\varphi^{-1}(x)yz$, then
$$\pi(g)=\pi(\varphi^{-1}(x))\pi(y)\pi(z)=x\cdot 1\cdot \pi(z)\in U;$$
and
$$p(g)=p(\varphi^{-1}(x))p(y)p(z)=q(x)\cdot y\cdot 1\in q(x)q(U)=q(xU).$$
Hence $W$ is a non-empty open set in $K$.

Moreover, $W\cap G^*$ is empty: if $h\in G^*$ and $\pi(h)\in U$, then $h=\varphi^{-1}(u)$ for some $u\in U$, hence
$$p(h)=p(\varphi^{-1}(u))=q(u)\in q(U).$$
Since $q(U)\cap q(xU)=\emptyset$, it follows that $p(h)\notin q(xU)$.
So $h\notin W$, that is, $W\cap G^*=\emptyset$.

While, $G^*$ is dense in $K$.
This contradiction confirms that
$q$ must indeed be discontinuous, thereby completing the proof.
\end{proof}

In the above theorem, we only consider precompact successors since there are nice structure theorems of compact groups. However, we do not know whether a compact group topology can have a successor which is not precompact.
Note that, in \cite{PHTX}, a key lemma for the characterization Fact \ref{AbelCom} says that a successor of a precompact group topology on an abelian group must be precompact.

\begin{question}
Does there exists a (connected) compact group $(G, \tau)$, such that $\tau$ admits a non-precompact successor?
\end{question}
The authors tend to believe that such groups do exist, since in \cite[Example 4.3]{HP}, an example is given to show that a precompact group topology can process non-precompact successors.

\subsection{Non-existence of successors of connected locally compact group topologies on solvable groups}

We are now going to prove our second main result.
To facilitate our exploration of group topologies, we first establish the following essential result, which elucidates the relationship between different group topologies. This lemma serves as a foundational tool in our subsequent proofs.

\begin{lemma}\label{finer}
Let $G$ be a group and $\tau,\sigma$ be group topologies on $G$.
If every non-empty $\sigma$-open set contains a non-empty $\tau$-interior, then $\sigma\subseteq \tau$.
\end{lemma}
\begin{proof}
Let $U$ be a $\sigma$-neighbourhood of $1$ and $V$ a $\tau$-open set with $V\subseteq U$.
Then $VV^{-1}$ is a $\tau$-neighbourhood of $1$ which is contained in $UU^{-1}$.
Since the family of sets of form $UU^{-1}$ is a local basis at $1$ for the topology $\sigma$, we obtain that $\sigma\subseteq \tau$.
\end{proof}
Next, we shift our focus to differentiable mappings and their properties. Consider the following setup:

Let $U$ be an open set of $\mathbb{R}^m$ and $f: U\to \mathbb{R}^n$ a differentiable mapping.
Then the differential of $f$ at a point $p\in U$, denoted by $\di f_p$, linearly maps $T_p U\cong \mathbb{R}^m$ into $\mathbb{R}^n$, where $T_p U=\{p\}\times\mathbb{R}^m$ is the tangent space of $U$ at $p$.
In the following, we shall always identify $T_pU$ with $\mathbb{R}^m$.
We refer the readers who are not familiar with these notions to \cite{Kli}.
\begin{lemma}\label{geo}
Let $U$ be a connected neighbourhood of the identity in $\mathbb{R}^m$ and $f: U\to \mathbb{R}^n$ be a smooth mapping, where $m,n$ are positive integers.
If $f(U)$ is not contained in a coset of any $(n-1)$-dimensional subspace, then
$$\underbrace{f(U)+f(U)+...+f(U)}_{n~\text{\em copies}}$$
contains a non-empty open set of $\mathbb{R}^n$.
\end{lemma}
\begin{proof}
We first claim that there exist
 $$p_1, p_2,...,p_n\in U~\mbox{and}~\mathbf{v}_i\in T_{p_i} U\cong \mathbb{R}^m ~(i=1,2,...,n)$$ such that $\{\di f_{p_i}(\mathbf{v_i}): i=1,2,...,n\}$ spans the whole space $\mathbb{R}^n$.
Suppose, for the sake of contradiction, that this is not the case.
Then there exists an $(n-1)$-dimensional subspace $R$ of $\mathbb{R}^n$ such that, for any $p\in U$ and $\mathbf{v}\in T_p U$, $\di f_p(\mathbf{v})$ is in $R$.

Take a 1-dimensional subspace $L$ of $\mathbb{R}^n$ such that $\mathbb{R}^n=R\oplus L$ and let $\pi: \mathbb{R}^n\to L$ be the corresponding projection.
Then $\pi\circ f$ can be viewed as a smooth real-valued function of $U$ and satisfies
$$\di (\pi\circ f)_p(\mathbf{v})=\di \pi_{f(p)}(\di f_p(\mathbf{v}))=\mathbf{0},$$
where $p\in U$ and $\mathbf{v}\in T_p U$. The last equality follows from that $\di f_p(\mathbf{v})\in R$ and $\di \pi_{f(p)}=\pi$.
This implies that the differential of $\pi\circ f$ is $0$ at every point of $U$, so $\pi\circ f$ is constant, as $U$ is connected.
Let $\{q\}=\pi\circ f(U)$. Then $f(U)\subseteq q+R$. This contradicts to our assumption on $f(U)$.

Now let $$\varphi:\underbrace{U\times U\times ... \times U}_{n~\text{copies}}\to \mathbb{R}^n, (x_1, x_2,...,x_n)\mapsto f(x_1)+f(x_2)+...+f(x_n).$$
Then, $$T_{(x_1, x_2,...,x_n)}(U\times U\times ... \times U)=T_{x_1}U\times T_{x_2}U\times...\times T_{x_n}U$$
 and
$$\di \varphi_{(x_1,x_2,...,x_n)}(\mathbf{u_1}, \mathbf{u_2},...,\mathbf{u_n})=\di f_{x_1}(\mathbf{u_1})+\di f_{x_2}(\mathbf{u_2})+...+\di f_{x_n}(\mathbf{u_n}).$$
In particular, we have
$$\di \varphi_{(p_1, p_2,...,p_n)}(\mathbf{0},...,\mathop{\mathbf{v}_i}\limits_{\text{the}~i-\text{th}},...,\mathbf{0})=\di f_{p_i}(\mathbf{v}_i).$$
So the image of $\di\varphi_{(p_1,p_2,...,p_n)}$ contains a basis of $\mathbb{R}^n$, i.e., $\di\varphi_{(p_1,p_2,...,p_n)}$ is surjective.
Then, there exists a neighbourhood $W$ of $f(p_1)+f(p_2)+...f(p_n)=\varphi((p_1,p_2,...,p_n))$ such that $\varphi$ maps some neighbourhood of
$(p_1,p_2,...,p_n)$ onto $W$ (see \cite[Theorem 0.5.2]{Kli}).
Thus,
$W$ is contained in $\varphi(U\times U\times...\times U)=f(U)+f(U)+...+f(U)$ and hence $f(p_1)+f(p_2)+...f(p_n)$ is an interior point of $f(U)+f(U)+...+f(U)$.
\end{proof}

\begin{proposition}\label{solLie}
Let $(G, \tau)$ be a solvable connected Lie group. Then $\tau$ admits no successors.
\end{proposition}
\begin{proof}
We prove by induction on the dimension of $(G, \tau)$.
When $\dim G=1$, $G$ is abelian, and the result follows trivially from Fact \ref{AbelLC}.

Now assume the proposition holds for groups of dimension  $\leq n-1$ and let $\dim G=n$.
We first consider the case where $G$ is simply connected. Let $N$ be a connected, closed, normal abelian group of $G$ with the smallest positive dimension (the existence of such a subgroup is guaranteed by Proposition \ref{abelnorm}).
Then $N$ is simply connected (see \cite[Theorem 11.2.15]{HN}), hence topologically isomorphic to an Euclidean space. Moreover, $(G/N, \tau/N)$ is a solvable connected Lie group of dimension $\leq n-1$.
By the induction hypothesis, $\tau/N$ has no successors. Thus, by Fact \ref{cent}, if $\sigma$ is a successor of $\tau$, then $\sigma/N=\tau/N$ and $\sigma\res_N\neq \tau\res_N$.

Now consider the action $$\varphi^*: G\times N\to N, (g, x)\to gxg^{-1}.$$
Since $N$ is abelian, $N$ is in the kernel of $\varphi^*$.
Hence $\varphi^*$ induces a group action $$\varphi: G/N \times N\to N.$$
If $\varphi$ is a trivial action, then so is $\varphi^*$, and thus $N$ is in the centre of $G$.
As we have already shown that $\sigma\res_N\neq \tau\res_N$, it follows then from Fact \ref{sub} (2) that $\sigma\res_N$ succeeds $\tau\res_N$.
However, this cannot happen by Fact \ref{AbelLC}, because $(N, \tau\res_N)$ is a connected locally abelian group. Therefore, the action $\varphi$ cannot be trivial.

We will denote $\varphi(g, x)$ by $g\cdot x$ for $g\in G/N$ and $x\in N$.
Moreover, $\varphi$ is smooth (resp. continuous) if $G/N$ and $N$ are given topologies induced from $\tau$ (resp. $\sigma$).
We shall denote the action by $\varphi_\tau$ and $\varphi_\sigma$, respectively, when $G/N$ and $N$ are given the corresponding pair of topologies.

For any $x\in N$, we shall denote by $\varphi_\tau^x$ or $\varphi_\sigma^x$ the mapping
$$G/N\to N, g\mapsto g\cdot x,$$
depends on the choice of the topologies. Thus, $\varphi^x_\tau$ is smooth and $\varphi^x_\sigma$ is continuous.

We consider the question in the following two cases. For the sake of convenience, we shall write the operation of the group $N$ additively and denote the identity by $0$.
\\

\noindent{\bf Case 1.} There exists $x\in N$ such that for any compact identity neighbourhood $V$ in $(G/N, \sigma/N)$, $\varphi_\sigma^x(V)$ is not contained in a coset of any proper linear subspace of $N$.
\begin{proof}[Proof of Case 1]
For any neighbourhood $U$ of $0$ in $(N, \sigma\res_N)$, one can find another identity neighbourhood $U'$ in the same group such that
$$\underbrace{U'+ U'+...+ U'}_{n~\text{copies}}\subseteq U.$$
By continuity of $\varphi^x_\sigma$, one may take an identity neighbourhood $V$ of $(G/N, \sigma/N)$ such that $\varphi_\sigma^x(V)\subseteq x+U'$.
Moreover, since $\sigma/N= \tau/N$ is a Lie group topology on $G/N$, we may assume that $V$ is diffeomorphic to a pathwise connected open set in $\mathbb{R}^m$, where $m$ is the dimension of $G/N$ .
By the assumption of this case, $\varphi_\sigma^x (V)$ is not contained in any coset of a proper linear subspace of $N$.
Note that $\varphi_\sigma^x$ and $\varphi_\tau^x$ are the same mapping if we forget the topologies.
So $\varphi_\tau^x(V)-x=\varphi_\sigma^x(V)-x\subseteq U'$.
Moreover, $V$ is open in $(G/N, \tau/N)$ since $\tau/N=\sigma/N$.
It then follows from Lemma \ref{geo} and the smoothness of $\varphi_\tau^x$ that
$$U\supseteq \underbrace{U'+ U'+... +U'}_{n~\text{copies}}\supseteq  \underbrace{\varphi_\tau^x(V)+...+\varphi_\tau^x(V)}_{n~\text{copies}}-nx$$
contains a non-empty open set in the Euclidean topology, i.e. $\tau\res_N$, on $N$.
Then, according to Lemma \ref{finer},  $\tau\res_N$ is finer than $\sigma\res_N$.
Since $\sigma\res_N$ is also finer than $\tau\res_N$, the two topologies coincide.
This is a contradiction.
\end{proof}

\noindent{\bf Case 2.} For any $x\in N$, there exists a compact identity neighbourhood $V$ in $(G/N, \sigma/N)$ such that $\varphi_\sigma^x(V)$ is contained in a coset space of a proper linear subspace of $N$.

\begin{proof}[Proof of Case 2]
Let $r=\dim N$.
For each $x\in N$, we take an identity neighbourhood $V_x$ in $(G/N, \sigma/N)$ such that, $N_x$, the linear subspace of $N$ generated by $\varphi_\sigma^x(V_x)-x$ is of the least dimension.
By the assumption, $\dim N_x\leq r-1$.
On the other hand, since the action $\varphi$ is non-trivial, there exists $x\in N$ such that $G/N$ does not fix $x$.
While, since $V_x$ generates $G/N$, this is saying that $x$ is not fixed by some element in $V_x$.
So, $k:=\dim N_x\geq 1$.

Now, take an identity neighbourhood $W$ in $(G/N, \sigma/N)$ such that $WW\subseteq V_x$.
By the choice of $V_x$, the set $\varphi_\sigma^x(W)-x$ also spans the vector space $N_x$.
Let $g_1, g_2,...,g_k\in W$ such that $\{\varphi_\sigma^x(g_1)-x, \varphi_\sigma^x(g_2)-x,...,\varphi_\sigma^x(g_k)-x\}$ is a basis of $N_x$.
Then, for any $h\in W$ and $i\in \{1,2,...,k\}$, one has
$$\varphi(h, \varphi_\sigma^x(g_i)-x)=\varphi(h, g_i\cdot x-x)=h\cdot(g_i\cdot x-x)=(hg_i)\cdot x-h\cdot x.$$
Since  $hg_i\in WW\subseteq V_x$, we obtain that $(hg_i)\cdot x\in \varphi_\sigma^x(V_x)\subseteq N_x+x$.
And it is evident that $h\cdot x$ is contained in $N_x+x$ as well.
So we have that $$\varphi(h, \varphi_\sigma^x(g_i)-x)=(hg_i)\cdot x-h\cdot x\in (N_x+x)-(N_x+x)=N_x.$$
Recall that $\{\varphi_\sigma^x(g_1)-x, \varphi_\sigma^x(g_2)-x,...,\varphi_\sigma^x(g_k)-x\}$ is a basis of $N_x$, so $N_x$ is an invariant subspace of $N$ for the linear automorphism $\varphi(h,-)$ for any $h\in W$.
Since $W$ generates the group $G/N$, $N_x$ is invariant under any element in $G/N$.
By the definition of the action $\varphi$, we known that $N_x$ is also invariant under $\varphi^*$.
That is, for any $h\in G$, we have $hN_xh^{-1}=N_x$.
So, $N_x$ is a normal subgroup of $G$.
While this contradicts to the choice of $N$.
So this case cannot happen.
\end{proof}

In summary, we have proved that $\tau$ has no successor if $(G, \tau)$ is simply connected and solvable.
Now let us consider the general case.
Let $(G, \tau)$ be a connected solvable Lie group and $(\widetilde{G}, \widetilde{\tau})$ its universal covering group.
Then $(G, \tau)$ is a quotient group of $(\widetilde{G}, \widetilde{\tau})$.
As $\widetilde{\tau}$ has no successor, so neither does $\tau$, by Corollary \ref{quot}.
\end{proof}

\begin{theorem}\label{solLC}
Let $(G, \tau)$ be a solvable connected locally compact group. Then $\tau$ admits no successors.
\end{theorem}
\begin{proof}
By the celebrated Gleason-Yamabe theorem, the connected locally compact group $G$ admits a compact normal subgroup $N$ such that $G/N$ is a Lie group.
Then, according to a theorem of Iwasawa \cite[Theorem 13]{Iwa}, $N$ is contained in a maximal compact subgroup $K$ of $G$, and $K$ is connected.
Since connected solvable compact groups are abelian \cite[Lemma 2.2]{Iwa}, $K$ and hence $N$ are abelian.

Now suppose that $\sigma$ is a successor of $\tau$.
On one hand, by Corollary \ref{connsub}, we have $\tau/N\prec \sigma/N$; on the other hand, our above proposition yields that $\tau/N$ admits no successors.
This contradiction completes the proof.
\end{proof}

\section*{Acknowledgements} We acknowledge Professor Wei He, Professor Jun Wang, and Doctor Xuchao Yao for their valuable suggestions. We also acknowledges the support from NSFC (grants No. 12301089, 12101445 and 12271258), as well as from the Natural Science Foundation of Jiangsu Higher Education Institutions of China (grant No. 23KJB110017) and  the support of training targets for outstanding young backbone teachers of the
``Qinglan Project'' in Jiangsu universities in 2023.

\end{document}